\documentclass[12pt]{amsart}

\usepackage{graphicx,psfrag,epsf}
\usepackage{enumerate}
\usepackage[margin=1in]{geometry}

\usepackage{graphicx}
\usepackage{amssymb, amsthm, amsfonts, amsmath}
\usepackage{epstopdf}
\usepackage{bbm}
\usepackage{algorithm}
\usepackage{algpseudocode}
\usepackage{hyperref}

\newtheorem{theorem}{Theorem}[section]

\newtheorem{remark}{Remark}[section]
\newtheorem{example}{Example}[section]
\newtheorem{lemma}{Lemma}[section]

\newtheorem{definition}{Definition}[section]

\def\soft{soft rejection sampling}
\def\WTGL{hard rejection sampling}
\def\WTGLC{Hard rejection sampling}
\def\return{\textbf{return} }
\def\restart{\textbf{restart}}
\def\assume{\textbf{Assumptions:} }
\def\X{ \textbf{X} }

\def\Y{ \textbf{Y} }

\def\WTGLTIME{ \text{\rm Time$_{\rm rej}$}}

\def\TIME{ \text{\rm Time}}
\def\speedup{\text{\rm speedup}}

\def\TIMEDDSH{\text{\rm Time$_{\rm DDSH}$}}
\def\TIMECDSH{\text{\rm Time$_{\rm CDSH}$}}

\def\XI{\textbf{$X^{(I)}$}}
\def\xI{\textbf{$x^{(I)}$}}
\def\EI{E^{(I)}}
\def\yI{y^{(I)}}
\def\tI{t_I}

\def\P{{\mathbb P}}

\def\PDCTSH{PDC deterministic second half}

\def\Var{\text{Var}}

\def\A{\mathcal{A}}
\def\B{\mathcal{B}}
\def\L{\mathcal{L}}
\def\R{\mathbb{R}}
\def\e{\mathbb{E} }

\def\bfa{{\bf a}}
\def\bfw{{\bf w}}

\def\XI{\textbf{$X^{(I)}$}}
\def\xI{\textbf{$x^{(I)}$}}
\def\EI{E^{(I)}}
\def\yI{y_I}

\def\ts{\hskip.03in}

\newcommand{\ignore}[1]{ }

\begin{document}
\title[PDC Deterministic Second Half]{Probabilistic divide-and-conquer: \\ deterministic second half}
\author{Stephen DeSalvo}
\address{UCLA Department of Mathematics, 520 Portola Plaza, Los Angeles, CA, 90095}
\email{stephendesalvo@math.ucla.edu}

\date{\today}

\begin{abstract}
We present a probabilistic divide-and-conquer (PDC) method for \emph{exact} sampling of conditional distributions of the form $\L( {\bf X}\, |\, {\bf X} \in E)$, where ${\bf X}$ is a random variable on $\mathcal{X}$, a complete, separable metric space, and event~$E$ with $\mathbb{P}(E) \geq 0$ is assumed to have sufficient regularity such that the conditional distribution exists and is unique up to almost sure equivalence. 
The PDC approach is to define a decomposition of $\mathcal{X}$ via sets $\mathcal{A}$ and $\mathcal{B}$ such that $\mathcal{X} = \mathcal{A} \times \mathcal{B}$, and sample from each separately. 
The deterministic second half approach is to select the sets $\mathcal{A}$ and $\mathcal{B}$ such that for each element $a\in \mathcal{A}$, there is only one element $b_a \in \mathcal{B}$ for which $(a,b_a)\in E$.  
We show how this simple approach provides non-trivial improvements to several conventional random sampling algorithms in combinatorics, and we demonstrate its versatility with applications to sampling from sufficiently regular conditional distributions. 
\end{abstract}

\maketitle


\section{Introduction}

The random sampling problem which motivates us is the following.
For any integer $n\geq 1$, let $\X_n = (X_1, X_2, \ldots, X_n)$ denote the $\R^n$--valued joint distribution with coordinates independent random variables with known marginal distributions~$\L(X_1)$, $\L(X_2)$, $\ldots$, $\L(X_n)$; denote by $\mathcal{F}$ the Borel sigma-algebra of measurable sets on $\R^n$.  
Given an event $E_n \in \mathcal{F}$, we define the distribution of $\X_n'$ as
\begin{equation}
\label{distribution}
\L(\X'_n) := \L\left((X_1, X_2, \ldots, X_n)\ \Big|\ \X_n \in E_n \right).
\end{equation}
Many random sampling problems of interest can be described in this manner; see for example \cite{ArratiaTavare, IPARCS, Duchon,Fristedt}.  

It shall be convenient to consider in these first few sections a more abstract version of the sampling problem, although we stress that one loses very little intuition by continuing to think in terms of~\eqref{distribution}: let $\mathcal{X}$ denote a metric space, $\mathcal{F}$ the Borel sigma-algebra of measurable sets on $\mathcal{X}$, and $\X$ an~$\mathcal{F}$-measurable random variable which induces a Radon (tight) probability measure $\lambda$ on $\mathcal{X}$.
Let $\mathcal{Y}$ denote another metric space, with $\mathcal{G}$ the Borel sigma-algebra of measurable sets on $\mathcal{Y}$. 
Suppose $T$ is a measurable map from $(\mathcal{X}, \mathcal{F})$ into $(\mathcal{Y}, \mathcal{G})$. 
Then, given the event $E = \{T = t\}$, we wish to sample from 
\begin{equation}\label{abstract:distribution} \L(\X') = \L(\, \X\, |\, T = t\, ). \end{equation}
In Section~\ref{example}, we give an example where this generality allows us to sample from a distribution which is not otherwise accessible by elementary conditioning. 

Next, we distinguish between \emph{approximate} and \emph{exact} sampling.  
By approximate sampling, we mean an algorithm such as a (forward) Markov chain (but \emph{not} coupling from the past~\cite{ProppWilson}) or a Gibbs sampler~\cite{GemanGeman}, which (unless one already starts with a random sample from the desired distribution) approaches the desired distribution ``in the limit," so that after any deterministically fixed, finite time there is an inherent amount of bias in the algorithm, albeit often acceptably small for many applications.  
By exact sampling, we mean an algorithm which terminates in finite time with probability 1 with a sample from the given distribution. 
See the recent book~\cite{huber2015perfect} for further examples of exact sampling, where it is also referred to as \emph{perfect simulation}.  \\

\begin{center}
 \emph{{\bf Goal:} Generate \emph{exact} samples as ``efficiently" as possible from }$\L(\X')$. \\ 
\end{center}
\vskip .15in

Our view of ``efficient" is the same as that specified in the introduction in \cite{Duchon}.  
That is, we address the following practical problem:  how much memory and time are required to sample from $\L(\X')$?  
We will assume that the dominant cost of an algorithm is the number of calls to a random number generator, which produces i.i.d.~uniform random values in the interval $(0,1)$, each with a cost which is $O(1)$; we also assume that arithmetic is largely negligible, and that precision is fixed.  
Finally, the runtimes of our algorithms will themselves be random variables, most with unbounded ranges, and so we only consider those algorithms which terminate in finite time with probability 1. 
As a result, we use the \emph{expected} runtime of an algorithm for the measure of efficiency.

When event $E$ has positive probability, the conditional distribution is well-defined and one can aways sample from the distribution $\L(\X)$ repeatedly until a variate generated lies in the set $E$, although the number of times to restart the algorithm may be impractical.  
We refer to this type of algorithm as~\emph{hard rejection sampling}, since each outcome is rejected with probability $0$ or $1$.  
A more general approach is to sample from a related distribution~$\L(\Y)$, and reject each outcome with a probability in $[0,1]$, which depends on $\L(\Y)$ and the observed variate; we refer to this type of algorithm as~\emph{soft rejection sampling}, since it requires some auxiliary randomness and does not always outright reject certain outcomes.  

In the case when event $E$ has probability 0, we are not aware of the existence of any analogous \WTGL\ algorithms which can be applied generally.  We are aware of several methods tailored for particular applications, for example sampling from the surface of an $n$--sphere, see for example \cite{Ball, Diaconis}, or convex polytope sampling, see for example \cite{devroye, Feller2}, but nothing as general as \WTGL\ which applies as readily and ubiquitously. 

At this point we also remind the reader that we are not considering algorithms which generate samples from an approximate distribution in finite time, such as standard (forward) Markov chains, or the Boltzmann sampler, which considers samples of $\L(\X)$ as acceptable surrogates for $\L(\X')$.  
Also included in this category are sampling algorithms which replace the event $E$ by some event $E'$ with positive probability, with $E \subset E'$. 
There is one exact sampling Markov chain approach we are aware of, coupling from the past \cite{ProppWilson}, which typically requires some type of monotonicity in order to apply practically. 

\emph{Probabilistic divide-and-conquer} (PDC) is a method for exact sampling which divides a sample space into two parts, samples each part separately, and then pieces them back together to form an exact sample~\cite{PDC}. 
That is, we take as our sample space some set $\mathcal{X}$, and decompose it by sets $\A$ and $\B$ such that $\mathcal{X} = \A \times \B$; we refer to this decomposition as the \emph{PDC division}. 
Then, for some random variable $\X$ on $\mathcal{X}$ and event~$E$, we assume that $\X$ can be written as $\X = (A,B)$, with random variables $A \in \A$ and $B \in \B$ such that $A$ and $B$ are independent, and  $\L(\X\, |\, E) = \L((A,B)\, |\, E)$. 
One then samples from $\L(A\, |\, E)$, say observing variate $x$, and then samples from $\L(B\, |\, E, A=x)$, say observing variate $y$; the PDC Lemma~\cite[Lemma~2.1]{PDC} affirms that $(x,y)$ is indeed an exact sample from $\L(\X\, |\, E)$ when event~$E$ has positive probability.  

Our purpose in this paper is two-fold. 
First, we prove in Lemma~\ref{PDC lemma} a generalization to the PDC Lemma~\cite[Lemma~2.1]{PDC}, which contains sufficient conditions for when PDC can be applied even when event $E$ has probability 0. 
The second is to expand on the utility of a particular PDC parameterization, originally exploited in~\cite[Section~3.3]{PDC} for the random sampling of integer partitions, dubbed \emph{deterministic second half}, which can (and should) be exploited almost as ubiquitously as hard rejection sampling.  
The name comes from the fact that we shall assume that sets $\A$ and $\B$ in the PDC division are chosen so that for each $a \in \A$, there is exactly one $b_a \in \B$ such that $(a,b_a) \in E$.

The outline of the paper is as follows.  
In Section~\ref{section:disintegration}, we state conditions under which the conditional distribution~\eqref{abstract:distribution} exists and is unique up to almost sure equivalence of probability distributions. 
In Section~\ref{PDC}, we prove an extended PDC lemma which allows for certain events of probability $0$, and provide a general PDC algorithm for the random sampling of distributions of the form in~\eqref{abstract:distribution}, stated in terms of quantities which may not be explicitly or efficiently computable in all situations, but which does \emph{not} yet assume a deterministic second half PDC division. 
In Section~\ref{PDCDSH}, we specialize to the deterministic second half case, and present two widely applicable PDC algorithms, one for discrete random variables and one for continuous random variables, which are stated in terms of explicitly computable quantities. 
In Section~\ref{cost}, we provide an analysis of the costs associated with the deterministic second half algorithm. 
In Section~\ref{example}, we give an example which demonstrates that PDC is \emph{not} simply rejection sampling, and which also highlights the subtleties involved with conditioning on events of probability~0. 
In Section~\ref{continuous examples}, we describe several applications of \PDCTSH\ to joint distributions of continuous random variables with event $E$ having probability 0. 
Finally, in Section~\ref{combinatorial classes}, we demonstrate the effectiveness of \PDCTSH\ for the random sampling of certain combinatorial distributions.

\section{Conditioning as disintegration}
\label{section:disintegration}
This section takes its name from~\cite{disintegration}, where an elegant, rigorous treatment of conditional distributions is presented: quoting from page 289, (with references for Tjur~\cite{tjur1974conditional} and Winter~\cite{winter1979alternate})
\begin{quotation}
It has long bothered us (and other authors, such as Tjur, 1974 and Winter, 1979) that there should be such a wide gap between intuition and rigor in conditioning arguments.  We feel that, in many statistical problems, manipulation of the conditional probability distribution is the most intuitive way to proceed.  However, we mathematical statisticians are trained to treat such conditional distributions with great caution, being aware of the menagerie of nasty counterexamples -- such as the Borel paradox -- that warn one away from conditional distributions.  \ldots

By way of a small amount of theory and a collection of illustrative examples, in this paper we present a case that disintegrations are easy to manipulate and that they recapture some of the intuition lost by the more abstract approach, allowing guilt-free manipulation of conditional distributions. 
\end{quotation}
We shall endeavor to recount the essential aspects of the framework presented in~\cite{disintegration} so that it is apparent what conditions are sufficient for a probabilistic divide-and-conquer algorithm to be fashioned, and highly recommend a thorough reading of~\cite{disintegration} for a more complete picture with enlightening examples. 
There is also a very reasonable topological restriction on the state space, which we emphasize is likely to be satisfied in most applications, in particular those involving real-valued random variables. 

We start with the definition of a disintegration, which is essentially a sigma-finite measure which acts like the elementary notion of a conditional distribution. 
We now introduce some definitions using standard terminology. 
Let $(\mathcal{X}, \mathcal{F})$ and $(\mathcal{Y}, \mathcal{G})$ denote two measurable spaces (where we have adopted standard nomenclature from~\cite{folland2013real}), and suppose $T$ is a measurable map from $(\mathcal{X}, \mathcal{F})$ into $(\mathcal{Y}, \mathcal{G})$. 
Suppose $\lambda$ is a probability measure on $(\mathcal{X}, \mathcal{F})$. 
Borrowing the notation from~\cite{disintegration}, we write $\lambda\,f = \int f(x) \lambda(dx)$, and for events $A \subset \mathcal{F}$ we let $\lambda (fA) = \int f(x) \mathbbm{1}\{x \in A\} \lambda(dx)$. 
Similarly, we denote the image of measure $\lambda$ by $T$ as $T\lambda$, with $(T\lambda)B = \lambda\, (T^{-1}A)$ denoting the probability of event $B \subset \mathcal{B}$ whenever $TA = B$. 

Next, we recall the essential characteristics of a conditional probability measure. 
First, we assume $\P$ is a probability measure on $(\mathcal{X},\mathcal{F}$). 
We next assume that $T$ only takes values in some finite subset $\mathcal{R} \subset \mathcal{Y}$ such that $\P\{T = t\} > 0$ for all $t \in \mathcal{R}$. 
Then the measure $\P(\, \cdot\, |\, T = t)$ satisfies
\begin{itemize}
\item[(a)] $\displaystyle \P(A\, |\, T = t) = \frac{\P(A \cap \{T = t\})}{\P\{T = t\}}$, for $A \in \mathcal{F}$ and $t \in \mathcal{R}$; \\
\item[(b)] $\P(\, \cdot\, |\, T=t)$ is a probability measure on $(\mathcal{X}, \mathcal{F})$ for all $t \in \mathcal{R}$; \\
\item[(c)] $\P(A) = \sum_{t \in \mathcal{R}} \P\{T=t\} \P(A\, |\, T=t)$, for $A \in \mathcal{F}$.  \\
\end{itemize}

Next is the definition of a disintegration. 
Let $\lambda$ be a sigma-finite measure on $\mathcal{F}$ and $\mu$ be a sigma-finite measure on $\mathcal{G}$.  

\begin{definition}[\cite{disintegration}]\label{definition:disintegration}
We say that $\lambda$ has a disintegration $\{\lambda_t\}$ with respect to $T$ and $\mu$, or a $(T,\mu)$-disintegration, if, for each nonnegative measurable function $f$ on $\mathcal{X}$:  \\
\begin{itemize}
\item[(i)] $\lambda_t$ is a probability measure on $\mathcal{F}$ concentrated on $\{T = t\}$; that is, $\lambda_t\{T \neq t\} = 0$ for $\mu$-almost all $t$; \\
\item[(ii)] $t \mapsto \lambda_t(f)$ is measurable; \\
\item[(iii)] $\lambda(f) = \int_t \lambda_t(f) d\mu(t)$. \\
\end{itemize}
\end{definition}

We shall see shortly that whenever the disintegration $\{\lambda_t\}$ consists of probability measures then we have $\mu = T\lambda$. 
Next, let us briefly explain sufficient conditions under which such a disintegration exists and is unique; see Theorem~\ref{existence:theorem} below. 
It is sufficient that the set $\mathcal{G}$ is countably generated and contains all singleton sets, and that $\lambda$ is a Radon measure on a metric space. 
A \emph{Radon measure} is a Borel measure for which $\lambda(K) < \infty$ for each compact $K$, and for any Borel set $B$ we define $\lambda(B) = \sup_{K \subseteq B} \lambda(K)$, where the supremum is over compact sets $K$.  
As pointed out in~\cite{disintegration}, by~\cite[Theorem~1.4]{BillingsleyConvergence}, a finite Borel measure on a complete, separable metric space is Radon (also known as \emph{tight}). 
Thus, we can specialize to many common probability measures on~$\R^n$ without difficulty. 

\begin{theorem}[{\cite[Existence Theorem]{disintegration}}]\label{existence:theorem}
Let $\lambda$ be a sigma-finite Radon measure on a metric space $\mathcal{X}$ and let $T$ be a measurable map from $(\mathcal{X}, \mathcal{F})$ into $(\mathcal{Y}, \mathcal{G})$.  Let $\mu$ be a sigma-finite measure on $\mathcal{G}$ that dominates the image measure $T \lambda$.  If $\mathcal{G}$ is countably generated and contains all the singleton sets $\{t\}$, then $\lambda$ has a $(T,\mu)$-disintegration.  The $\lambda_t$ measures are uniquely determined up to an almost sure equivalence: if $\{\lambda_t^\ast\}$ is another $(T,\mu)$-disintegration then $\mu\{t \in \mathcal{Y} : \lambda_t \neq \lambda_t^\ast\} = 0$.
\end{theorem}

The next two theorems now complete the picture by providing the conditions for which the disintegration consists of probability measures, and also when the densities exist. 

\begin{theorem}[\cite{disintegration}]\label{theorem:2}
Let $\lambda$ have a $(T,\mu)$-disintegration $\{\lambda_t\}$, with $\lambda$ and $\mu$ each sigma-finite. 
\begin{itemize}
\item[(i)] The image measure $T\lambda$ is absolutely continuous with respect to $\mu$, with density $\lambda_t \mathcal{X}$. 
\item[(ii)] The measures $\{\lambda_t\}$ are finite for $\mu$-almost all $t$ if and only if $T\lambda$ is sigma-finite. 
\item [(iii)] The measures $\{\lambda_t\}$ are probabilities for $\mu$-almost all $t$ if and only if $\mu = T\lambda$.
\item[(iv)] If $T\lambda$ is sigma-finite then $(T\lambda)\{\lambda_t \mathcal{X} = 0\} = 0$ and $(T\lambda)\{\lambda_t \mathcal{X} = \infty\} = 0$.  For $T\lambda$-almost all $t$, the measures 
\[ \tilde \lambda_t(\cdot) = \frac{\lambda_t(\cdot)}{\lambda_t\mathcal{X}}\left\{0 < \lambda_t \mathcal{X} < \infty\right\} \]
are probabilities that give a $(T, T\lambda)$-disintegration of $\lambda$. 
\end{itemize}
\end{theorem}

\begin{theorem}[\cite{disintegration}]\label{theorem:3}
Let $\lambda$ have a $(T,\mu)$-disintegration $\{\lambda_t\}$, and let $\rho$ be absolutely continuous with respect to $\lambda$ with a finite density $r(x)$, with each of $\lambda$, $\mu$, and $\rho$ sigma-finite. 
\begin{itemize}
\item[(i)] The measure $\rho$ has a $(T,\mu)$-disintegration $\{\rho_t\}$ where each $\rho_t$ is dominated by the corresponding $\lambda_t$, with density $r(x)$.  
\item[(ii)] The image measure $T \rho$ is absolutely continuous with respect to $\mu$, with density $\lambda_t r$. 
\item[(iii)] The  measures $\{\rho_t\}$ are finite for $\mu$-almost all $t$ if and only if $T\rho$ is sigma-finite. 
\item [(iv)] The measures $\{\rho_t\}$ are probabilities for $\mu$-almost all $t$ if and only if $\mu = T\rho$.
\item[(v)] If $T\rho$ is sigma-finite then $(T\rho)\{\lambda_t r = 0\} = 0$ and $(T\rho)\{\lambda_t r = \infty\} = 0$.  For $T\rho$-almost all $t$, the measures defined by 
\begin{equation}\label{density} \tilde \rho_t(f) = \frac{\lambda_t(f\, r)}{\lambda_t(r) }\left\{0 < \lambda_t(r) < \infty\right\} \end{equation}
are probabilities that give a $(T, T\rho)$-disintegration of $\rho$. 
\end{itemize}
\end{theorem}

Again quoting from~\cite{disintegration}:
\begin{quotation}
The simple formula~\eqref{density} is the general version of the familiar method for calculating conditional densities as a ratio of joint density to marginal density.  It is more useful than the familiar formula because it does not require the conditioning variable to be a coordinate projection on a Euclidean space with Lebesgue measure playing the role of $\lambda$.  
\end{quotation}

\section{A PDC lemma with disintegrations}
\label{PDC}

We start with a sample space $\mathcal{X}$ with Borel sigma-algebra $\mathcal{F}$; random variable $\X \in \mathcal{F}$ which induces a probability measure $\P$ on $\mathcal{X}$; $\mathcal{F}$-measurable event~$E$; and we wish to sample from the distribution in~\eqref{abstract:distribution}. 
We write the sample space as $\mathcal{X} = \A \times \B$, with random variables $A \in \A$ and $B \in \B$ defined such that they are independent, and~\eqref{abstract:distribution} can be written as $\L( (A,B)\, |\, (A,B) \in E)$. 
The most generic approach to sampling from the distribution~\eqref{abstract:distribution} is \WTGL, 
given by Algorithm~\ref{WTGL procedure}.  The generic PDC random sampling algorithm is given in Algorithm~\ref{PDC procedure}.
Note that, when $\P(\X \in E) =0$, hard rejection sampling does not terminate in finite time with probability 1. 

\begin{algorithm}{\rm
\begin{algorithmic}
\State 1.  Generate sample from $\L(A)$, call it $a$.
\State 2.  Generate sample from $\L(B)$, call it $b$.
\State 3.  Check if $(a,b) \in E$; if so, return $(a,b)$, otherwise restart.
\end{algorithmic}}
\caption{\WTGLC\ of $\L(\, (A,B)\, |\, (A,B) \in E\, )$}
\label{WTGL procedure}
\end{algorithm}

\begin{algorithm}{\rm
\begin{algorithmic}
\State 1. Generate sample from $\L(A\, |\, E)$, call it $x$.
\State 2. Generate sample from $\L(B\, |\, E, A=x)$ call it $y$.
\State 3. Return $(x,y)$.
\end{algorithmic}}
\caption{Probabilistic Divide-and-Conquer sampling of $\L(\, (A,B)\, |\, (A,B) \in E\, )$} 
\label{PDC procedure}
\end{algorithm}

The PDC Lemma~\cite[Lemma 2.2]{PDC} affirms that the resulting pair $(x,y)$ from Algorithm~\ref{PDC procedure} is an exact sample from $\L(\X')$ when $\P(\X \in E)>0$.  
We now generalize this lemma to include certain events for which $\P(\X \in E) = 0$. 
Recall the notation from the previous section: 
Let $(\mathcal{X}, \mathcal{F})$ and $(\mathcal{Y}, \mathcal{G})$ denote two measurable spaces, and suppose $T$ is a measurable map from $(\mathcal{X}, \mathcal{F})$ into $(\mathcal{Y}, \mathcal{G})$. 

\begin{lemma}\label{PDC lemma}
Suppose 
\begin{itemize}
\item[(i)] $\P$ is a given probability Radon measure on $(\mathcal{X}, \mathcal{F})$;  
\item[(ii)] the image measure $T\ts \P$ is a probability measure on $(\mathcal{Y}, \mathcal{G})$, where $\mathcal{G}$ is countably generated and contains all the singleton sets $\{t\}$. 
\end{itemize}
Then $\P$ has disintegration probability measure $\P_t := \P(\, \cdot\, |\, T = t)$, which is unique up to almost sure equivalence of probability measures, and for events $E$ of the form $E = \{T = t\}$, we have $\L(\X')$ is well-defined and unique on events of positive $\P_t$ measure.  

Furthermore, assume 
\begin{itemize}
\item[(a)] there are metric spaces $\A$ and $\B$ such that $\mathcal{X} = \A \times \B$;
\item[(b)] there are probability measures $\P_\A := \pi_\A \P$ and $\P_\B := \pi_\B \P$, such that $\P = \P_\A \times \P_\B$, corresponding to measurable random variables $A \in\A$ and $B\in\B$ which are independent with distributions also denoted by $\L(A)$ and $\L(B)$, respectively;
\item[(c)] for each $a \in \mathcal{A}$, the map $T_a(\cdot) := T(a, \cdot)$ is measurable (i.e., $T_a$ is the map $T$ holding the coordinate $\A$ fixed), and we let $t_a$ denote the set of values such that $\{T=t, A=a\} = \{T_a \in t_a\}$ holds. 
\end{itemize}
Then the following distributions exist and are unique up to almost sure equivalence:
\[\qquad  \L(X) := \L( A\, |\, T = t), \qquad \L(Y\, |\, X=x) := \L(B\, |\, T = t, A = a) = \L( B\, |\, T_a(B) \in t_a); \]
Finally, we have $\L(X,Y) = \L(\X')$. 
\end{lemma}
\begin{proof}
First, we demonstrate that all distributions exist and are unique up to almost sure equivalence. 
That $\P$ has disintegration probability measure $\P_t$ is a consequence of Theorem~\ref{theorem:2}, which also implies that $\L(\X')$ is well-defined and unique up to almost sure equivalence. 
By projecting $\X$ onto $\A$, we conclude the same holds for $\L(X)$ (see for example~\cite[Example~2]{disintegration}), and denote this measure by $\P_X$. 
By the same argument, $\L(B\, |\, T=t)$ also exists and is unique up to almost sure equivalence. 
Then, under the assumption that $T_a$ is measurable for each $a \in \A$, there exists a disintegration probability measure for $\L(B\, |\, T=t)$ under map $T_a$, i.e., $\L(B\, |\, T=t, A = x)$ for each $x \in A$, which is precisely $\L(Y\, |\, X=x)$; we denote this measure by $\P_{Y,x}$. 

Next, we show $\L(X,Y) = \L(\X')$.  
For all nonnegative measurable $f : \mathcal{A} \times \mathcal{B} \to \R$, by property (iii) of Definition~\ref{definition:disintegration}, applied once to $\P$ and a second time to $\P_t$ (see also~\cite[Example~4]{disintegration}), we have 
\begin{align*}
 \int f(A,B)\, (\P_\A \times \P_\B)(dA\times dB) & = \int_t\ \left(\int f(A,B)\ \P_t(dA\times dB) \right) (T\P)(dt) =  \\
               & = \int_t\ \left(\int \left(\int f(A,B)\ \P_{Y,x}(dB) \right) \P_X(dA) \right) (T\P)(dt) =  \\
               & = \int_t\ \left(\int f(X,Y)\ (\P_X\times \P_Y)(dX\times dY) \right) (T\P)(dt).
\end{align*}
Then, for all bounded measurable $g:  \mathcal{A} \times \mathcal{B} \to \R$, we separate $g$ into its positive and negative parts and apply the same argument, which completes the proof.
\end{proof}

Before we specialize to a more concrete setting, let us first introduce and extend a particularly fruitful PDC algorithm which uses soft rejection sampling, see~\cite[Algoirthm~3]{PDC}. 
Recall that soft rejection sampling samples a related distribution $\L(\Y)$, and introduces an auxiliary random variable $U$, uniform over the interval $(0,1)$, which is used to reject samples of $\L(\Y)$ so that they appear in their correct proportion in $\L(\X')$.  
In our application, we will apply soft rejection sampling to $\L(A\, |\, T=t)$ using samples from $\L(A)$.  
We define the function~$q(a)$, $a \in \A$,  as the correct proportion for using soft rejection sampling in this manner. 

We then define the \emph{rejection function} as 
\begin{equation}\label{eqt} s(a) := \frac{q(a)}{\sup_{\ell\in\A} q(\ell)}, \qquad a \in \A. \end{equation}
(In the event that $q(a) = 0$ for all $a\in \A$, we assign $s(a) = 0$.)
In other words, we will sample from the distribution $\L( (A,U)\, |\, U < s(a) )$, where $s : \A \to [0,1]$ is a measurable function, and with $s$ chosen as in Equation~\eqref{eqt}, the first coordinate has distribution $\L(A\, |\, T=t)$. 
Note, however, that we need some additional assumptions so that Algorithm~\ref{PDC procedure von Neumann} below terminates in finite time with probability 1:
\begin{equation}
\tag{A1}\label{A1} 0<\sup_{\ell \in \A} q(\ell) < \infty; 
\end{equation}
\begin{equation}
\tag{A2}\label{A2} \P_\A\{a \in \A : q(a) > 0\} > 0.
\end{equation}

\begin{algorithm}[H]{\rm
\begin{algorithmic}
\State 0.  Assume the conditions of Lemma~\ref{PDC lemma}, and in addition assume~\eqref{A1} and~\eqref{A2}. 
\State 1. Generate sample from $\L(A),$ call it $a$.
\State 2. Reject $a$ with probability $1-s(a)$, where $s(a)$ is given in~\eqref{eqt}; 
 otherwise, restart. 
\State 3. Generate sample from $\L(B\, |\, (a,B) \in E),$ call it $y$.
\State 4. Return $(a,y)$.
\end{algorithmic}}
\caption{PDC sampling from $\L( (A,B)\, |\, (A,B)\in E)$ using soft rejection sampling}
\label{PDC procedure von Neumann}
\end{algorithm}

\begin{remark}{\rm
In the event that $\sup_{\ell \in \A} q(\ell)$ is not practical to compute efficiently, we may replace this quantity in Equation~\eqref{eqt} with \emph{any} upper bound $C$ satisfying $\sup_{\ell \in \A} q(\ell)\leq C<\infty$, and Algorithm~\ref{PDC procedure von Neumann} is still an exact sampling algorithm.  Of course, as is well known (see for example~\cite{Rejection}), the efficiency of the algorithm diminishes the larger $C$ is taken, in the sense that the expected number of rejections before we accept a sample is multiplied by the corresponding quotient $C / \sup_{\ell \in \A} q(\ell)$, and so it is optimal to take $C= \sup_{\ell \in \A} q(\ell)$. 
}\end{remark}

Of course, there is still the matter of how to sample from $\L(B\, |\, (a,B) \in E)$ for each $a \in \A$. 
An often optimal divide-and-conquer strategy fashions sets $\A$ and $\B$ so that $\L(B\, |\, (a,B) \in E)$ is a smaller version of the original sampling problem $\L(\X\, |\, \X \in E)$, split roughly in half, and recursively repeats until a simple base case is reached; this is referred to as self-similar, recursive PDC in~\cite[Section~3.5]{PDC}, and it was used to yield an asymptotically efficient sampling algorithm for the random sampling of integer partitions. 
The fashioning of a self-similar PDC algorithm, however, requires the efficient computation of the rejection function $s(a)$, which can be computed efficiently in the case of integer partitions, see Remark~\ref{self-similar}; in general, however, the efficient computation of $s(a)$ is not a trivial matter, which is why we specialize to the more practical deterministic second half setting.

\section{PDC with deterministic second half}
\label{PDCDSH}
The efficient computability of $s(a)$ in Algorithm~\ref{PDC procedure von Neumann} is a vital consideration when fashioning a PDC division. 
It is with this consideration in mind that we revisit a PDC division first introduced in~\cite[Section~3.3]{PDC} called deterministic second half.
For each $a\in \A$, define the set $E_a := \{b \in \B: (a,B) \in E\}$ as the set of all $b \in \B$ which can be paired with a given $a\in \A$ and satisfy $(a,b) \in E$. 
Suppose $\A$ and $\B$ are chosen so that for each $a\in \A$, we have $|E_a| = 1$. 
This type of PDC division was shown to provide surprisingly large speedups to hard rejection sampling for many interesting examples in~\cite[Section~3.3]{PDC}, while keeping $s(a)$ practical to compute. 
We now generalize this approach.  

In addition to the assumptions in Lemma~\ref{PDC lemma}, we now add assumption~\eqref{DSH}, which makes the distribution $\L(B\,|\,(a,B)\in E)$ trivial; 
that is, the second stage of the algorithm is completed by a uniquely determined value. 
\begin{equation}\label{DSH}
\tag{DSH}
\mbox{For each $a \in \A$, \quad $|E_{a}| = |\{b_a\}| = 1$.}
\end{equation}
To reiterate, the assumption~\eqref{DSH} is an abbreviation for \emph{\underline{d}eterministic \underline{s}econd \underline{h}alf}, and is not required for PDC in general, nor for Lemma~\ref{PDC lemma}. 
In terms of Algorithm~\ref{PDC procedure von Neumann} it is provably faster than hard rejection sampling (with respect to a costing scheme which we make precise in Section~\ref{cost}), and in many cases of interest the rejection function is explicit and efficient to compute. 

At this point we also specialize to the parameterization in~\eqref{distribution}, since it is indicative of many interesting examples. 
That is, we assume a sample space $\mathcal{X} = \R^n$, with $\X = (X_1, X_2, \ldots, X_n)$ a joint distribution of real-valued \emph{independent} random variables, where each $X_i$ has a corresponding marginal distribution $\L(X_i)$, sometimes written as $\P_{X_i}$. 
In addition, we assume that $T \equiv T(X_1, \ldots, X_n)$ is some measurable function which satisfies the conditions of Lemma~\ref{PDC lemma}.  

\begin{example}\label{convolution:example}
Suppose $X_1, X_2, \ldots, X_n$ are discrete random variables. Let $T = \sum_{i=1}^n X_i$, so that $T\ts \P = \P_{X_1} \ast \P_{X_2} \ast \cdots \ast \P_{X_n}$ is the convolution of measures. 
Then for any fixed $t\in \R$ in the range of $T$, taking~$E = \{T = t\}$, we have that the distribution~\eqref{distribution} is well-defined by Theorem~\ref{existence:theorem}. 
Since $t$ is uniquely determined given any $n-1$ of the variables $X_1, \ldots, X_n$, a valid PDC deterministic second half approach would be to select an index $i$, say, e.g., $i=1$, and let $\A = \R^{n-1}$, $\B = \R$, with random variables $A = (X_2, \ldots, X_n)$ and $B = X_1$.  
Applying Algorithm~\ref{PDC procedure von Neumann}, the rejection function is given by 
\[ s((y_2, \ldots, y_{n})) = \frac{\P\left(X_1 =  t-\sum_{i=2}^{n} y_i\right)}{\sup_{\ell}\P(X_1 = \ell)}, \qquad (y_2, \ldots, y_{n}) \in \A. \]
\end{example}

\begin{example}\label{convolution:example:continuous}
Suppose in Example~\ref{convolution:example} we instead take $X_1, X_2, \ldots, X_n$ to be continuous random variables with a density, say we denote $f_{X_1}$ for the density of $X_1$. 
Applying Algorithm~\ref{PDC procedure von Neumann}, the rejection function is then given by 
\[ s((y_2, \ldots, y_{n})) = \frac{f_{X_1}(t-\sum_{i=2}^{n} y_i)}{\sup_{\ell}f_{X_1}(\ell)}, \qquad (y_2, \ldots, y_{n}) \in \A. \]
\end{example}

Examples~\ref{convolution:example} and~\ref{convolution:example:continuous} form the foundation of our intuition, but the form of $T$ is needlessly specific, and so we conclude the theoretical treatment of PDC deterministic second half with a generalization below that we have found both concrete and adequate for almost all applications encountered.  Even though, as mentioned in the quotation at the end of Section~\ref{section:disintegration}, we are not restricted to coordinate projections and Euclidean spaces, many natural PDC divisions are in fact coordinate projections on $\R^k$, and so we have endeavored to make this final treatment the most widely accessible and applicable. 

To simplify notation, when the values of the parameters are understood from context we adopt the conventions $\X' \equiv \X_n'$, $E \equiv E_n$, $\X \equiv \X_n \equiv (X_1, \ldots, X_n)$, $[n] = \{1,\ldots, n\}$.  
In order to describe the class of divisions, we now state the following definitions.  

\begin{definition}\label{main:definition}
Let $I = \{i_1, i_2, \ldots \} \subset [n]$ denote some subset of indices, with $i := |I|$. 
\begin{enumerate}
\item Let $X_I := \pi_I(\X) = (X_i)_{i\in I}$ denote the $\R^{i}$--valued projection;
\item let $\XI := \pi_{[n]\setminus I}(\X)$ denote the $\R^{n-i}$--valued projection;
\item let $\sigma_I : \R^{n-i} \times \R^{i} \to \R^n$ denote the (unique) rearrangement of the combined set of elements in $x$ and $y$ in the following manner: for $x = (x_1, \ldots, x_{n-i}) \in \R^{n-i}$ and $y = (y_1, \ldots, y_{i} \in \R^{i}$, $z = \sigma_I(x,y)$ is the vector for which $z_{i_j} = y_j$ for $j=1,\ldots, i$, and which keeps the original order of $x_1, \ldots, x_{n-i}$ in $z$, i.e., for all indices $1 \leq j_1 < j_2 \leq n-i$ for which  $x_{j_1} = z_{\ell_1}$ and $x_{j_2}= z_{\ell_2}$, we have $1 \leq \ell_1 < \ell_2 \leq n$.  In other words, $\sigma_I$ is defined so that we have $\{X \in E\} = \{\sigma_I( \XI, X_I ) \in E\}$;
\item define the \emph{$I$-completable-set} of $E_n$ as 
\[\EI :=  \pi_{[n]\setminus I}^{-1}(E_n)  =  \{ x\in \R^{n-i} : \exists y \in \R^{i} \mbox{ such that } \sigma_I(x,y) \in E_n \};
\]
\item define the \emph{$I$-section of $E_n$ given $\xI\in \EI$} as 
\[E_I \equiv E_{I|\xI} :=  \pi_I^{-1}(E_n\, |\,\xI)  =  \{ y\in \R^{i} : \sigma_I(\xI,y) \in E_n \};
\]
\item define $T_I \equiv T_I(X_I | \xI) := T\vert_{E_{I|\xI}}(\sigma_I(\xI, X_I)). $ 
\end{enumerate}
\end{definition}

To aid in the presentation of our algorithms, we also define the `cemetery' state $\Delta$ (see, e.g.,~\cite[Section~3.1]{DurrettBook}) to be such that $S \Delta = \Delta$ for all transformations $S$, with $\P(X = \Delta) = 0$ and the density function of $X$, if it exists, evaluated at $\Delta$ equals 0, for all random variables $X$. 
Under Assumption~\eqref{DSH}, for each $\xI \in \EI$, we define $\yI \equiv \yI(\xI) \in E_I$ to be the unique completion such that $\sigma_I(\xI,y_I) \in E$, and for each $\xI \notin \EI$, we define $y_I(\xI) = \Delta$. 
In addition, the random variable $T_I$ is one-to-one on $E_I$, and so we define $t_I \equiv t_I(\xI) := T_I\, \yI$ as the unique image of the point $\yI$ under map~$T_I$. 
\emph{We shall use the above definitions of $y_I$ and $t_I$ in all subsequent analysis and examples whenever assumption~\eqref{DSH} is in effect.}

We now present two concrete applications of PDC with deterministic second half we believe are of most use in practical applications: one for discrete random variables and the other for continuous random variables. 
See Section~\ref{example} for an application which uses Algorithm~\ref{PDC procedure} directly. 
For continuous random variables, we now utilize~Theorem~\ref{theorem:3}, which implies the existence of the density of a disintegration under conditions which will be satisfied for our application. 
We will denote the density of a random variable $X$ by $f_X$.  
 In what follows, $U$ will denote a uniform random variable between $0$ and $1$, independent of all other random variables, and $u$ an observed variate from this distribution.

\begin{algorithm}
\caption{PDC with deterministic second half for discrete random variables} 
\begin{algorithmic}
\State {\bf Input:} 
\begin{itemize}
\item[] Distributions $\L(X_1), \ldots, \L(X_n)$;
\item[] measurable function $T$, $t \in \mbox{range}(T)$; 
\item[] index set~$I \subset \{1, \ldots, n\}$. 
\end{itemize}
\State {\bf Output:} 
\begin{itemize}
\item[] A sample from $\L( (X_1, \ldots, X_n) | T=t).$ 
\end{itemize}
\State \assume 
\begin{itemize}
\item[] \eqref{A1}, \eqref{A2}, \eqref{DSH}, and $X_1, \ldots, X_n, T$ are discrete random variables.
\end{itemize} \\
\State Sample from $\L(\XI)$, denote the observation by $\xI$.
\If {$\xI \in \EI$ \mbox{ and } $u < \frac{\P\left(X_I = \yI\right)}{\max_\ell \P(X_I = \ell)}$ }
		\State \return $\sigma_I(\xI,\yI)$
\Else
\State \restart
\EndIf \\
\end{algorithmic} \label{PDC discrete}
\end{algorithm}

\begin{algorithm}
\caption{PDC with deterministic second half for continuous random variables}
\begin{algorithmic}
\State {\bf Input:} 
\begin{itemize}
\item[] Distributions $\L(X_1), \ldots, \L(X_n)$;
\item[] measurable function $T$, $t \in \mbox{range}(T)$; 
\item[] index set~$I \subset \{1, \ldots, n\}$. 
\end{itemize}
\State {\bf Output:} 
\begin{itemize}
\item[] A sample from $\L( (X_1, \ldots, X_n) | T=t).$ 
\end{itemize}
\State \assume 
\begin{itemize}
\item[] \eqref{A1}, \eqref{A2}, \eqref{DSH}, and $X_1, \ldots, X_n, T$ are continuous random variables. 
\end{itemize} \\
\State Sample from $\L(\XI)$, denote the observation by $\xI$.
\If {$\xI \in \EI$ \mbox{ and } $u < \frac{f_{T_I}\left(\tI\right)}{\sup_\ell f_{T_I}(\ell)}$ }
		\State \return $(\xI,\yI)$
\Else
\State \restart
\EndIf
\end{algorithmic} \label{PDC continuous}
\end{algorithm}

Each of these algorithms follows by Lemma~\ref{PDC lemma} and Algorithm~\ref{PDC procedure von Neumann}. 
However, we shall write out explicitly the steps with which to derive the rejection function in order to draw attention to a key difference between the discrete and continuous versions; see also Remark~\ref{key:difference} below.   
Note that while our rejection probability in Algorithm~\ref{PDC continuous} is stated in terms of the distribution $\L(T_I\, |\, \XI=\xI)$, we return the value $\yI$, which is the corresponding value in the range of $X_I$.  

\begin{theorem}\label{discrete theorem}
Algorithm~\ref{PDC discrete} samples from the distribution given in Equation~\eqref{distribution}.
\end{theorem}

\begin{proof}
We demonstrate that $ \frac{P\left(X_I = \yI\right)}{\max_\ell(P(X_I = \ell))}$ is the right proportion.

We let $h$ denote the probability mass function of $\L( \XI\, |\, E_n )$, and $g$ denote the probability mass function of $\L( \XI )$.  Our rejection proportion is of the form: suppose we observe state~$j$ under the distribution $g$, then we reject if
\[ u > \frac{h(j)}{C g(j)}, \]
where $C$ is any constant such that
\[h(\ell) \leq C g(\ell), \qquad \text{ for all states $\ell$}. \]
The quantity $C$ is the expected number of iterations of the acceptance/rejection procedure before we accept a sample, see for example \cite{devroye}, and so in particular we would like to find the smallest $C$.  Since our distributions are already specified, we obtain
\[ \frac{1}{C} = \min_\ell \frac{g(\ell)}{h(\ell)} = \min_\ell \frac{\P(T=k)}{\P(T=k | \XI = \ell)} = \frac{\P(T=k)}{\max_{\ell} \P(T_I(X_I|\ell) = t_I(\ell) | \XI = \ell )}, \]
and our rejection step reduces to 
\[ u > \frac{h(j)}{C g(j)} =  \frac{\P(T_I(X_I|j) = \tI(j))}{\max_{\ell} \P(T_I(X_I | \ell) = \tI(\ell))}.  
\]
By assumption~(DSH), once we accept $\xI$, the completion $\yI$ is unique.   Since $T_I$ is discrete, we have
\[ \frac{\P(T_I = \tI(j))}{\max_{\ell} \P(T_I = \tI(\ell))} = \frac{\P(X_I = \yI(j))}{\max_\ell \P(X_I = \ell)}.   \qedhere
\]
\end{proof}

\begin{theorem}\label{continuous theorem}
Algorithm~\ref{PDC continuous} samples from the distribution given in Equation~\eqref{distribution}.
\end{theorem}

\begin{proof}
The proof follows in a similar manner as the proof of Theorem~\ref{discrete theorem}, with probabilities replaced with probability density functions \emph{where appropriate}, which are guaranteed to exist by Theorem~\ref{theorem:3} since we assume $T$ is continuous. 
We have
\[ \frac{1}{C} = \inf_\ell \frac{g(\ell)}{h(\ell)} = \inf_\ell \frac{f_T(k)}{f_{T_I|\ell}(\tI(\ell))} = \frac{f_T(k)}{\sup_\ell f_{T_I|\ell}(\tI(\ell))}, \]
and the rejection step reduces to 
\[ u > \frac{h(j)}{C g(j)} = \frac{f_{T_I|j}(\tI(j))}{\sup_\ell f_{T_I|\ell}(\tI(\ell))}. \qedhere \]
\end{proof}

\begin{remark}\label{key:difference}{\rm
There is an important difference between algorithms~\ref{PDC discrete} and~\ref{PDC continuous}. 
When $X_I$ and $T_I$ are discrete, the random variables can be rearranged \emph{before} the soft rejection step from $\{T_I = \tI(\xI)\}$ to $\{X_I = \yI(\xI)\}$, whereas when $X_I$ and $T_I$ are continuous, one would have to apply a standard change of variables formula in order to determine the rejection probability in terms of the density of $X_I$; see for example Section~\ref{sphere}. 
}\end{remark}

\section{Speedup analysis}
\label{cost}

  The \WTGL\ acceptance condition is $\{\X \in E\}$, which in the discrete setting may be written as 
\[ \{ \text{$\XI \in \EI$ and $U < P(X_I = \yI)$}\}. \]
The acceptance condition for Algorithm~\ref{PDC discrete} is 
\begin{equation}\label{PDCTSH discrete}
\left\{ \text{$\XI \in \EI$ and $U < \frac{\P( X_I = \yI)}{\max_\ell \P(X_I = \ell)}$}\right\}.
\end{equation}
It is easy to see that there exists a coupling of the random variables such that all events giving an acceptance in the \WTGL\ algorithm would also be accepted in the \PDCTSH\ algorithm.  The added efficiency in this case comes from the use of soft rejection sampling, which enlarges the space by the factor $\max_\ell \P(X_I = \ell)^{-1}$.  

For continuous random variables, \soft\ transforms the otherwise expected infinite-time \WTGL\ algorithm into an acceptance condition of the form 
\begin{equation}\label{PDCTSH continuous}
\left\{ \text{$\XI \in \EI$ and $U < \frac{f_{T_I|\XI}(\tI)}{\sup_\ell f_{T_I|\XI}(\ell)}$}\right\},
\end{equation}
which is an event of \emph{positive} probability under assumptions~\eqref{A1},~\eqref{A2},~\eqref{DSH}.

Since the memory requirements for \WTGL\ and the algorithms presented are on the same order of magnitude, we focus solely on run--time.  
We assume that arithmetic operations are negligible, whereas the cost of generating a single random uniform variate from a given interval is $O(1)$, and the cost of generating $n$ independent uniform random variables from a given interval is $O(n)$, regardless of the magnitudes of the values\footnote{One might think of this as fixed floating-point precision implemented on a computer. }.  
Implicitly, we also assume, quite critically, but also quite reasonably for many applications, that computing $\yI$ and $t_I$ is always $O(1)$; that is, we assume that completing the sample is indeed a trivial matter. 

\begin{definition}
   For a given algorithm $P$, which generates a sample from a distribution $\L(\X)$, we denote the \emph{expected} time of completion by $\TIME_P(\X)$.  When there is no subscript, we assume there is a default direct sampling method available.  
   
   When $P$ is the \WTGL\ algorithm, we denote the expected time of completion by \WTGLTIME$(\X')$.  
When $0<\TIME_P(\X)<\infty$, the \speedup\ of algorithm $P$ relative to \WTGL\ is defined by
\begin{equation}\label{speedup}
\speedup :=  \frac{\WTGLTIME(\X')}{\TIME_P(\X')}.
\end{equation}
\end{definition}

\begin{remark}{\rm
We have 
\begin{align*}
\WTGLTIME(\X') & = O\left(\frac{\TIME(\X)}{\P(\X \in E)}\right) .
\end{align*}
For all algorithms $P$ with $0<\TIME(P) < \infty$, we have  $\speedup\in(0,\infty]$, with $\speedup>1$ representing an improvement to \WTGL.
In the case when \WTGL\ does not terminate in finite time with probability 1, we write $\speedup = \infty$.  
}\end{remark}

\begin{theorem}\label{discrete speedup theorem}
Let \TIMEDDSH$(\X')$ denote the expected run-time of Algorithm~\ref{PDC discrete}.  We have 
\[ \TIMEDDSH(\X') = O\left(\WTGLTIME(\X') \max_\ell \P(X_I = \ell)\right), \]
whence
\begin{equation}\label{discrete speedup}
 \speedup = \Omega\left( \left(\max_\ell \P(X_I=\ell)\right)^{-1} \right). 
 \end{equation}
\end{theorem}

\begin{proof}
Recall the optimal value of $C$ is given by
\[ C = \frac{\max_\ell \P(X_I =\ell)}{\P(T=k)}. \]
The cost of this algorithm with the optimal $C$ is then
\begin{align*}
 \TIME\left( \XI  \right) \, C  & = O\left(\frac{\TIME\left(\XI\right)}{P(T=k)} \max_\ell \P(X_I = \ell)\right) \\
   & = O\left(\WTGLTIME\left(\X'\right) \max_\ell \P(X_I = \ell)\right),
\end{align*}
which by Equation~\eqref{speedup} implies the speedup is $\Omega(\max_\ell \P(X_I = \ell))^{-1}$.
\end{proof}

\begin{remark}\label{optimal:I:discrete}{\rm 
Theorem~\ref{discrete speedup theorem} indicates that the optimal choice of $I$ in Algorithm~\ref{PDC discrete} is one that minimizes the maximal point mass in the distribution of $X_I$.  
}\end{remark}

\begin{theorem}\label{continuous PDC theorem}
Let \TIMECDSH$(\X')$ denote the expected run-time of Algorithm~\ref{PDC continuous}.  We have 
\[ \TIMECDSH(\X') = O\left( \WTGLTIME\left(\XI|\EI\right)\ \sup_\ell f_{T_I|\ell}(\tI(\ell))\right). \]
Also,  $\speedup = \infty$. 
\end{theorem}
\begin{proof}
The optimal value of $C$ is given by
\[ C = \frac{\sup_\ell f_{T_I|\ell}(\tI(\ell))}{f_T(k)}. \]
The cost of this algorithm with the optimal $C$ is then
\begin{align*}
 \TIME\left( \XI  \right) \, C  & = O\left(\TIME\left(\XI\right)\frac{\sup_\ell f_{T_I|\ell}(\tI(\ell))}{f_T(k)} \right) \\
   & = O\left(\WTGLTIME\left(\XI|\EI\right)\sup_\ell f_{T_I|\ell}(\tI(\ell)) \right).
\end{align*}
By assumption, $E = \{T = t\}$ for some continuous random variable $T$ with a density and $t \in \mbox{range}(T)$, and so $\P(\X \in E) = 0$, whence, since Algorithm~\ref{PDC continuous} has finite expected time, we have $\speedup = \infty$. 
\end{proof}

\section{An illustrative example}
\label{example}
This section highlights the fact that PDC is not simply rejection sampling. 
It also motivates the use of disintegrations from Section~\ref{section:disintegration} rather than elementary conditioning. 

An example which demonstrates that care must be taken when conditioning on events of probability 0 is given by the following example of \cite{Proschan} (see also \cite[Section 4.9.3]{CasellaBerger}).  The problem stated on a particular exam is as follows: \emph{If $U$ and $V$ are independent standard normals, what is the conditional distribution of $V$ given that $V = U$?}   
The three distinct ways in which this problem was solved started with the following joint distributions: \\ \\
\begin{tabular}{llcl}
(1) $\bigl(\ (U,V)$&$\vert$ & $U-V=0$& $\bigr)$; \\
(2) $\bigl(\ (U,V)$&$\vert$ &  $\frac{U}{V} = 1$& $\bigr)$; \\
(3) $\bigl(\ (U,V)$&$\vert$ &  $\mathbbm{1}(U=V)$ & $\bigr)$.
\end{tabular} \\ \\
In our notation, this is the same as $((U,V)\, |\, T = t)$, where  \\ \\
\begin{tabular}{lll}
(1) $T = U-V$ & and & $t=0$; \\
(2) $T = V/U$ & and & $t=1$; \\
(3) $T = \mathbbm{1}(U=V)$ &and& $t=1$.
\end{tabular} \\ \\
To apply Algorithm~\ref{PDC continuous}, we first sample $U$ from a standard normal distribution, and apply a rejection depending on the distribution $\L(T_I\, |\, \XI = a)$; we have \\ \\
\begin{tabular}{ll}
(1) $T_I = V-a$, & reject if $u > e^{-a^2/2}$; \\
(2) $T_I = a/V$, &  reject if $u > |a|\, e^{-a^2} / \sqrt{2\, e}$; \\
(3) $T_I = \mathbbm{1}(V=a)$, &Not applicable.
\end{tabular} \\ \\
Note that the last case, $T = \mathbbm{1}(U=V)$, does not satisfy the assumptions of Algorithm~\ref{PDC discrete} or Algorithm~\ref{PDC continuous}; that is, $\P(\X \in E)$ is not an event of positive probability, nor is $T$ a random variable with a density.
Rather than applying Algorithm~\ref{PDC continuous}, we can instead determine the conditional distribution $\L(V \, |\, T=t)$ directly, as the original problem demands, sample according to that distribution, and then appeal to Lemma~\ref{PDC lemma} and Algorithm~\ref{PDC procedure} directly. \\ \\
\begin{tabular}{lll}
(1) $T_I = V-a$, & $f_{V | T=0}(v) = e^{-v^2}/\sqrt{\pi}$, & $-\infty < v < \infty$; \\
(2) $T_I = a/V$, &  $f_{V | T=1}(v) = |v| e^{-v^2}$, & $-\infty < v < \infty$; \\
(3) $T_I = \mathbbm{1}(V=a)$, & $f_{V | T = 1}(v) = e^{-v^2/2}/\sqrt{2\pi},$ &  $-\infty < v < \infty$.  
\end{tabular} \\ \\
In each of these cases, the density $f_{V|T}$ is guaranteed to exist by Theorem~\ref{theorem:3}, and we have $\L(V\, |\, T=t, U=a)$ is a point mass at $a$, $a \in \R$.

\section{Applications of a theoretical nature}
\label{continuous examples}

\subsection{Uniform weight}
\label{uniform}

This section contains a compelling application of PDC deterministic second half, one which is implicit in many other applications.  
That is, when the random variable $T_I$ assigns the same weight to $\tI$ for all $\xI \in \EI$, then each sample generated from $\L(\XI)$ is accepted with the same proportion, which can be scaled up to 1, giving a rejection probability of 0.

\begin{theorem}\label{uniform theorem}
  Algorithm~\ref{PDC Uniform} generates a sample from the distribution $\L(\X')$, with expected runtime $O(\TIME(\XI)\,~\P(\XI\in\EI)^{-1}).$ 
\end{theorem}
\begin{proof}
If $\XI \notin\EI$, we reject with probability 1.  
Assuming $\XI \in \EI$, the rejection function is given by 
\[s(a) = \frac{q(a)}{\sup_{\ell \in \A} q(\ell)} = 1, \] 
thus we reject with probability 0 any sample that is completable.  
The expected runtime is therefore the inverse of the probability of generating a completable sample, i.e., $\P(\XI \in \EI)^{-1}$, times the cost to generate a sample from $\L(\XI)$. 
\end{proof}

\begin{algorithm} 
\caption{PDC deterministic second half with uniform weight}
\begin{algorithmic}
\State {\bf Input:} 
\begin{itemize}
\item[] Distributions $\L(X_1), \ldots, \L(X_n)$;
\item[] measurable function $T$, $t \in \mbox{range}(T)$; 
\item[] index set~$I \subset \{1, \ldots, n\}$. 
\end{itemize}
\State {\bf Output:} 
\begin{itemize}
\item[] A sample from $\L( (X_1, \ldots, X_n) | T=t).$ 
\end{itemize}
\State \assume 
\begin{itemize}
\item[] \eqref{A1}, \eqref{A2}, \eqref{DSH}; 
\item[] $q(a) = q(b)$ for all $a,b \in \EI$. 
\end{itemize} \\
\State Sample from $\L(\XI)$, denote the observation by $\xI$.
\If {$\xI \in \EI}$
\State \return $\sigma_I(\xI,\yI)$
\Else
\State \restart
\EndIf
\end{algorithmic} \label{PDC Uniform}
\end{algorithm}

See Section~\ref{polytope} for an example involving sums of independent uniform random variables. 

\subsection{Exponential Distribution}

When the $X_i$, $i=1,\ldots,n$, are independent and exponentially distributed random variables with parameters $\lambda_i > 0$, the marginal density function of $X_i$ is given by 
\[ f_{X_i}(x) = \lambda_i e^{-\lambda_i x}, \qquad x>0. \]
When the event $E$ is of the form $E = \{\sum_{i=1}^n X_i = k\}$, then $\P(E) = 0$, and there is no \WTGL\ algorithm in general.  
We take $I = \{i\}$, then since the density $f_{T_i}$ is bounded by $\lambda_i$, we apply Algorithm~\ref{PDC continuous}.  
The acceptance condition is thus
\[ \left\{ \text{\rm $\sum_{j\neq i} X_j \leq k$ and $U < e^{-\lambda_i y_I}$ }\right\}, \]
where $y_I = k - \sum_{j\neq i}x_j$.  
By Theorem~\ref{continuous PDC theorem}, we have
\[ \TIMECDSH(\X') = O\left(\frac{\lambda_i}{\P(\sum_{j\neq i} X_j \leq k)} \right).\]

\subsection{Beta Distribution}

A continuous random variable $X$ is said to have Beta$(\alpha, \beta)$ distribution, $\alpha>0, \beta>0$, if it has density
\[ f_X(x) = c_{\alpha,\beta}\, x^{\alpha-1} (1-x)^{\beta-1}, \qquad 0<x<1,\]
where $c_{\alpha,\beta}$ is the normalization constant.  When at least one of $\alpha, \beta$ is less than 1,  the density $f_X(x)$ is not bounded.  
When both $\alpha$ and $\beta$ are at least 1, then we have 
\[ \max_x f_X(x) = \frac{\alpha - 1}{\alpha+\beta-2}, \qquad \alpha,\beta > 1. \]
Thus, if we consider $X_1, X_2, \ldots, X_n$ independent Beta$(\alpha_j, \beta_j)$, $j=1,\ldots,n$, with $E = \{\sum_{j=1}^n X_j = k\}$, then as long as there exists an index $i$ such that both $\alpha_i$ and $\beta_i$ are greater than 1,  we can apply Algorithm~\ref{PDC continuous}.  The acceptance condition is given by
\[ \left\{ \text{\rm $\sum_{j\neq i} X_i \in [k,k-1]$ and $U < \frac{ c_{\alpha,\beta}(\alpha+\beta-2)}{\alpha-1} y_I^{\alpha-1}(1-y_I)^{\beta-1}$ }\right\}, \]
where $y_I = k - \sum_{j\neq i}x_i$.

\subsection{Small Ball Probabilities}

Suppose $X_i$, $i\geq 1$ are i.i.d.~with distribution $\P(X_i=1) = \P(X_i=-1) = \frac{1}{2}$.  Let $w_i$ denote real--valued weights with $|w_i| \geq 1$, $i\geq 1$.  Define $T := \sum_{i=1}^n w_i X_i$.  Then for some open set $G$, $\P(T \in G)$ is known as the small ball probability, see for example \cite{Nguyen}.  To obtain sample paths, the simplest approach is to apply \WTGL.  However, noting that $X_i$ is actually a discrete uniform distribution over the set $\{-1,1\}$, using Algorithm~\ref{PDC Uniform} we can apply \PDCTSH\ and select any index $I$, and sample until $\{\XI \in G^{(I)}\}$, where $G^{(I)} = \bigcup_{g \in G} (g+w_I) \cup (g-w_I)$.  

When $G$ is an open set of length $2r$, then it was shown in \cite{Erdos} that $P(T \in G)$ is at most $2^{-n}$ times the sum of the largest $r$ binomial coefficients in $n$.  Let us assume that the values $w_i$ are all integer--valued, and $G = (-1,1)$.  Let $S_r(n)$ denote the sum of the largest $r$ binomial coefficients in $n$.  We have
\[ \P\left(T \in (-1,1)\right) \leq 2^{-n} S_1(n) = 2^{-n} \binom{n}{\lfloor n/2\rfloor}. \]
If we apply \PDCTSH, this becomes
\[ \P\left(\XI \in (-w_I-1,-w_I+1) \cup (w_I-1, w_I+1)\right) \leq 2^{-(n-1)} S_2(n-1) = 2^{-(n-1)} \binom{n}{\lfloor n/2\rfloor}, \]
which saves at most an anticipated factor of 2.  

Also, as was exploited in \cite{RBM}, if in addition there exist two distinct elements $w_j \neq w_\ell$, then we let $I = \{j,\ell\}$, and we have the range of $w_j X_j + w_\ell X_\ell$ is uniform over four distinct elements, say $\{v_1, v_2, v_3, v_4\}$; let $V_i = (v_i-1, v_i+1)$, then $V := \cup_{i=1}^4 V_i$ is an open set of length $8$, whence 
\[ \P(T - w_jX_j - w_\ell X_\ell \in V) \leq 2^{-(n-2)}S_4(n-2). \]
One can keep going with this idea.  If we let $I = \{j_1, j_2, \ldots\}$, then we must have that $\sum_{\ell} w_{j_\ell}X_{j_\ell}$ is uniform over distinct elements; i.e., each combination of $\pm 1$ in $X_{j_\ell}$ must yield a distinct element for the sum.  This is true, e.g., if $w_{j_1} = 1$, $w_{j_2} = 2$, \ldots, $w_{j_\ell} = 2^\ell$.

\subsection{Sampling from the surface of the $n$--sphere}\label{sphere}
The following is~\emph{not} an example of PDC deterministic second half, but serves to illustrate the versatility of the PDC approach. 
Consider the distribution
\[ \left(X_1, \ldots, X_n\, |\, X_1^2 + \ldots + X_n^2 = k\right), \]
where all random variables are continuous, which corresponds to some distribution on the surface of an $n$--sphere.  It is known how to obtain the uniform distribution and certain other distributions over the surface of an $n$--sphere, see e.g. \cite{devroye, Diaconis} (see also~\cite{Ball} for a generalization to the $\ell_p^n$-ball); however, if we change the form of the conditioning slightly (for example, replace $X_1^2$ with $X_1$ in the conditioning event), or place a particular demand on any of the marginal laws of $X_1, \ldots, X_n$, then these techniques do not generalize in a straightforward manner. 
PDC, on the other hand, is robust with respect to small changes. 

We take $I = \{i\}$, for some $i\in \{1,\ldots, n\}$.  Then we have 
\[ {\L(T_I\, |\, \XI=\xI, E) = \L(X_i^2\, |\, X_i^2 = t_I(\xI))}.\]
The rejection has the form: 
\[ \left\{y_I \in \text{range}(X_i^2) \text{ and } U < \frac{\frac{1}{2\sqrt{\tI}}\left( f_{X_i}(\sqrt{\tI}) + f_{X_i}(-\sqrt{\tI})\right)}{\sup_\ell f_{X_i^2}(\ell)} \right\}. \]
Note that we have calculated explicitly the transformation from the distribution $\L(T_I\, |\, \XI)$ to the distribution $\L(X_I)$.  This is \emph{almost} \PDCTSH, see~\cite{DeSalvoSudoku, DeSalvoImprovements}, because there are actually two possible values for $X_I$ to complete the sample, even though the distribution $\L(T_I\, |\, \XI)$ is trivial. 
Once a sample is accepted, we simply choose an outcome, $\sqrt{\tI}$ or $-\sqrt{\tI},$ in proportion to its value determined by the density function $f_{X_i}$.

This example also illustrates why the deterministic second half condition~\eqref{DSH} is a statement about the number of ways to complete a sample given $\XI \in \EI,$ rather than a statement about the triviality of the distribution $\L(T_I\, |\, \XI)$.

\subsection{Uniform Spacings}

Suppose we place $m$ points uniformly distributed over the interval $[0,1]$, call them $u_1, u_2, \ldots, u_m$.  Let $u_{(1)}$, $u_{(2)}$, \ldots, $u_{(m)}$ denote the ordering of the points in ascending order.  Then it is well--known, see for example \cite{devroye, Feller2}, that the marginal distributions are given by 
\[ u_{(i)} \sim {\rm Beta}\left( i, m+1-i\right), \qquad i=1,\ldots,n. \]
The differences between consecutive points in the interval are in fact i.i.d.~with (taking $u_{(0)}=0, u_{(m+1)})=1$)
\[ u_{(i)}-u_{(i-1)} \stackrel{D}{=} \frac{E_i}{\sum_{i=1}^{m+1} E_i}, \qquad i=1,\ldots,m+1,\]
 where $E_i$ are exponential distributions with parameter 1.  In other words, to obtain a sample from $(E_1, E_2, \ldots, E_n\, |\, \sum_{i=1}^n E_i = 1)$, one can follow the steps above, which does not use PDC.  
Of course, as alluded to in Section~\ref{sphere}, any deviation from this very specific form of distribution renders this approach effectively useless, whereas PDC can still be applied.
 
\subsection{Convex Polytope Sampling}
\label{polytope}
Suppose we wish to sample from a convex polytope $P\subset \R^n$ with vertices $\{v_1, \ldots, v_m\}$.  Assuming none of the points $v_i$ are degenerate, i.e., there do not exist any $v_i$ such that $v_i \in \text{ConvexHull}\{v_1, \ldots, v_{i-1},v_{i+1},\ldots,v_m\}$, we can sample uniformly from $P$ via Algorithm~\ref{Polytope} (see \cite{Feller2}; see also \cite{devroye}).

\begin{algorithm}
\caption{\cite{Feller2} Convex polytope sampling} 
\begin{algorithmic}
\State \assume $P$ is a convex polytope with $m$ vertices. 
\State Generate $(U_1, U_2, \ldots, U_{m-1})$ i.i.d in the interval $[0,1]$, denoted by $(u_1, \ldots, u_{m-1}).$
\State Sort the points and denote them as $u_{(1)},\ldots,u_{(m-1)}.$  Let $u_{(0)} = 0$ and $u_{(m)} = 1$. 
\State Let $y_i = u_{(i)}-u_{(i-1)}$ for $i=1, \ldots, m$. 
\State \return  $\sum_{i=1}^m y_i\,  v_i$.
\end{algorithmic}
\label{Polytope}
\end{algorithm}

The key aspect of this algorithm is that it has a time and memory requirement on the order of $m$, the number of vertices, which means that it is not efficient for polytopes with a large number of vertices compared to their dimension, which we now demonstrate. 

The hypersimplex $H_{n,k}$ is defined as 
\[ H_{n,k} = \left\{ (x_1, \ldots, x_n) \in [0,1]^n : \sum_{i=1}^n x_i = k \right\}. \]
We can obtain a random point inside the hypersimplex using Section~\ref{uniform}.    Each coordinate is uniformly distributed over the interval $[0,1]$, so by Theorem~\ref{uniform theorem}, the \PDCTSH\ algorithm is simply to sample $(u_2, \ldots, u_n)$ from independent uniform distributions over $[0,1]$ until $u_1 := k - \sum_{i=2}^n u_i \in [0,1]$.

The Permutahedron $P_n$, see e.g.,~\cite{Postnikov}, is the convex hull of all $m = n!$ permutations of the coordinates of the point $(1,2,\ldots, n)$.  It can be described as follows: 
\[ P_n = \left\{(x_1, \ldots, x_n) \in [1,n]^n : \sum_{i=1}^n x_i = \binom{n+1}{2}, R \right\}, \]
where $R$ denotes Rado's condition~\cite{Rado}, which is 
\[ R = \{\text{for all $j\geq 1$, } x_{(n)}+\ldots + x_{(j)} \leq n+(n-1)+\ldots+j\}.\]
  Let
\[ Q_n = \left\{(x_1, \ldots, x_n) \in [1,n]^n : \sum_{i=1}^n x_i = \binom{n+1}{2} \right\}. \]
This is a scaled version of the hypersimplex, which can be sampled using Section~\ref{uniform} where each coordinate is uniformly distributed over the interval $[1,n]$.  Its asymptotic volume is given in \cite{Eulerian} as
\[ \text{\rm Vol}(Q_n) \sim (n-1)^n \sqrt{\frac{6}{\pi n}}. \]
It is well--known that $\text{\rm Vol}(P_n) = n^{n-2}$, see for example~\cite[Proposition~2.4]{Postnikov}, i.e., the number of forests on $n$ labeled vertices.  Hence, the probability that a point in $Q_n$ is also in $P_n$ is given by
\[ \P(R | Q_n) \sim \frac{n^{n-2}}{ (n-1)^n\sqrt{\frac{6}{\pi n}}} \sim n^{-3/2}\, e\, \sqrt{\frac{\pi}{6}}. \]
Thus, to sample points from inside the permutahedron, one can first sample points from inside the scaled hypersimplex using Algorithm~\ref{PDC Uniform}, and then apply \WTGL\ with respect to Rado's condition $R$.

\section{Applications to combinatorial classes}
\label{combinatorial classes}

\subsection{Table Methods}
In any discussion of random sampling of combinatorial structures, invariably one is led to the recursive method of Nijenhuis and Wilf \cite{NW}, which uses a table of values to calculate conditional probability distributions based on recursive properties of a corresponding combinatorial sequence. 
This method has several costs:
\begin{itemize}
\item[1.] Computational cost to create the table
\item[2.] Storage cost to store the table
\item[3.] Computational cost to generate samples from the table.
\end{itemize}
If one is capable of handling items 1 and 2, then the table methods are typically the fastest known methods for random generation, since they provide fast lookups equivalent to unranking algorithms; a good survey of these and other similar algorithms is \cite{StantonWhite}; see also~\cite{denise1999uniform}.  
An extensive treatment of how to apply PDC in this case is given in~\cite{DeSalvoImprovements}, which offers a larger speedup at the expense of creating a partial table, and applies the recursive method to a smaller-sized set of objects. 
The advantage of our current approach is that it is table-free and simple to implement, while still offering an improvement over hard rejection sampling.

\subsection{Assemblies, Multisets, and Selections}\label{sect:iparcs}
  Let $w: \mathbb{N}\to \mathbb{R}$ denote a weighting function, and define the sequence $w_i := w(i)$, $i\geq 1$, and we interpret $a \cdot b$ as the usual dot product.  
For the remainder of this section, in order to keep the notation in \cite{IPARCS}, we define ${\bf Z} \equiv \X$ and let $T = {\bf Z}\cdot w$.   
In the examples that follow, the random vector {\bf Z} is discrete, hence the event $\{T = t\}$ has strictly positive probability for each $t \in \text{range}(T)$.  
     
A well-known example is that of the cycle decomposition of a permutation. 
Letting $C_i(n)$ denote the number of cycles of length $i$ in a random permutation of $n$, it was shown in \cite{SheppLloyd} that the joint distribution of all cycle lengths satisfies
\[ \L(C_1(n), \ldots, C_n(n)) = \L\left( (Z_1, \ldots, Z_n) \bigg| \sum_{i=1}^n i Z_i = n\right), \]
where $Z_1, \ldots, Z_n$ are independent Poisson random variables with $\e Z_i = 1/i,$ $i=1,2,\ldots,n$. 

There are many combinatorial distributions that can be described in this fashion.  
       In \cite{IPARCS}, a unified framework is presented which is applied to three main types of combinatorial structures.  We recount some of the basic definitions, and refer the interested reader to \cite{IPARCS} for a more inspiring exposition.
     
     For $\bfa = (a_1, a_2, \ldots, a_n)$ a sequence of nonnegative integers, and weights $\bfw = (w_1, \ldots, w_n)$, let $N(n,\bfa, \bfw)$ denote the number of combinatorial objects of weight $n$ having $a_i$ components of size~$w_i$, $i=1,\ldots,n$.  Since each component has size~$w_i$, the total contribution to the weight of the object by component $i$ is $w_i\, a_i$, and summing over all $i$ gives us the total weight $n = \sum_{i=1}^n w_i\,a_i$ of the object.   Suppose $J \subset \{1,\ldots, n\}$.  The examples of interest will have the following form:
\[ N(n,\bfa, \bfw) = 1(\bfw\cdot \bfa = n) f(J, n) \prod_{i\in I} g_i(a_i), \]
for some functions $f$ and $g_i$, $i\in J$, with
\[ p(n) = \sum_{\bfa \in \mathbb{Z}_+^n} N(n,\bfa,\bfw) \]
denoting the total number of objects of weight $n$.  We now suppose that our combinatorial objects are chosen uniformly at random.  Then the number of components of size~$i$ is a random variable, say with distribution $C_i$, $i=1,\ldots,n$, and $C = (C_1, \ldots, C_n)$ is the joint distribution of \emph{dependent} random component sizes that satisfies $C \cdot \bfw= n$.  The distribution of $C$ is given by
\begin{equation}\label{C distribution}
 \P(C = \bfa) = 1(\bfw \cdot \bfa = n) \frac{f(J,n)}{p(n)} \prod_{i\in J} g_i(a_i).
 \end{equation}
For each $x>0$, let independent random variables $Z_i$, $i\in J$, have distributions
\begin{equation}\label{Z distribution}
 \P(Z_i = k) = c_i(x)\, g_i(k)\, x^{w_ik}, \qquad i\in J,
 \end{equation}
where $c_i$, $i\in J$, are the normalization constants, given by
\[ c_i = \left( \sum_{k \geq 0} g_i(k) x^{w_ik} \right)^{-1}. \]
Now we can state the following theorem.

     \begin{theorem}[\cite{IPARCS}]
Assume $J \subset \{1,\ldots, n\}$.  Let $Z_J = (Z_i)_{i\in J}$ denote a vector of independent random variables with distributions given by Equation~\eqref{Z distribution}.  Let $C_J = (C_i)_{i\in J}$ denote the stochastic process of random component sizes with distribution given by Equation~\eqref{C distribution}.  Then
\[ C_J =^d (Z_J | T = n). \]
Furthermore,
\begin{equation}\label{Tn}
 \P(T = n) = \frac{p(n)}{f(J,n)} x^n \prod_{i\in J} c_i(x). 
 \end{equation}
     \end{theorem}

\begin{remark}\label{always use PDC} {\rm 
The \WTGL\ algorithm for such combinatorial classes has cost
\[ \WTGLTIME(Z_J) = O\left( \TIME(Z_J)\, \frac{f(J,n)}{p(n)} x^{-n} \prod_{i\in J} c_i(x)^{-1} \right). \]
The form of the condition $\{T=n\}$ implies that the \PDCTSH\ algorithm has $|I|=1$, i.e., $I$ consists of a single index.  By Theorem~\ref{discrete speedup theorem}, we have for any $i \in J$,
\[ \speedup = \Omega\left(\max_{k} c_i(x)\, g_i(k)\, x^{w_ik}\right)^{-1}. \]
The optimal choice of $i$ is one that minimizes this maximal probability.  In fact, \emph{any choice of $i$} will provide a speedup, which is an even more compelling reason to use PDC in this setting.  
}\end{remark}

\subsection{The Boltzmann Sampler}
\label{Boltzmann}
The Boltzmann sampler is a popular approach to random generation of combinatorial structures based on generating functions; see \cite{DuchonF, Boltzmann, Flajolet}.  In terms of the structures introduced in Section~\ref{sect:iparcs}, the Boltzmann sampler is equivalent to sampling from the distribution $\L(\X)$, and accepting \emph{all} samples, regardless of whether they lie in the set $E$. 
While this is not exact sampling, for the purposes of obtaining large scale characteristics, this approach has proven to be fruitful and reasonably accurate for many statistics of interest; see \cite{duchon2011random} and the references therein.

In terms of exact sampling, the typical recommended procedure is to perform hard rejection sampling on Boltzmann samples; see \cite[Section 5.2]{duchon2011random}. 
Recently, several authors have noted various ways in which to obtain exact sampling using Boltzmann sampling; see \cite{PrePDC, binomial}, based on an original approach used in \cite{alonso}. 
The approach is equivalent to self-similar PDC (see \cite{PDC} and also Remark~\ref{self-similar}), where one divides the sample space into two roughly equal-sized parts (in terms of randomness, not necessarily in terms of component-sizes), each a recursively defined copy of the original, samples each separately, and then pieces them together.  
The main caveat is the ability to compute the appropriate rejection probabilities, which are often in a form which can be approximated by local central limit theorems. 
In the case of Motzkin words in \cite{alonso}, the rejection probabilities were computable since they were written as a quotient of certain binomial coefficients. 

While we envision the ultimate goal of random sampling of combinatorial structures to be self-similar PDC, the computing of rejection probabilities is a major impediment. 
Once such probabilities become accessible, PDC deterministic second half may not be as competitive. 
Until then, we stress that the main utility of PDC deterministic second half is the fact that \emph{no local limit theorems or detailed knowledge of the sample space is required knowledge}, only the explicit calculation of the distribution and its maximum in~\eqref{eqt}; see Remark~\ref{always use PDC}.

 \subsection{Example: Integer Partitions}
 \label{sect:ip}
     A random unrestricted integer partition of (non--random) size~$n$ is described by $J = \{1,\ldots,n\}$, $w(i) = i$, $f(J,n) = 1, g_i(a_i) = 1$, $p(n)$ is the number of partitions of $n$, usually denoted by $p(n)$.  Hence, for each $0<x<1$, we have normalization factors $c_i = (1-x^i)$, so Equation~\eqref{Z distribution} specializes to
\[ \P(Z_i = k) = x^{i k}(1-x^i), \qquad 0<x<1, \ i=1,\ldots n, \]
i.e., $Z_i$ is geometrically distributed with parameter $1-x^i$.  
  In this example, $Z_i$ denotes the number of parts of size $i$ in a random partition (of random size).
  
  The probability that we generate a particular partition of $n$, with multiplicities $(c_1, \ldots, c_n)$,  is given by 
\[ \P(Z_1=c_1, \ldots, Z_n = c_n) = \prod_{i=1}^n \P(Z_i = c_i) = \prod_{i=1}^n x^{i\, c_i} (1-x^i) = x^{\sum_{i=1}^n i\, c_i} \prod_{i=1}^n (1-x^i). \]
Then, since each partition of $n$ satisfies $\sum_{i=1}^n i\, c_i = n$, we have
\[ \P(Z_1=c_1, \ldots, Z_n = c_n) = x^n \prod_{i=1}^n (1-x^i). \]
In other words, only the size of the partition determines its likelihood of being generated using this approach, and all partitions of the same size are equally likely to appear.  Thus, either by summing over all partitions of $n$, or by Equation~\eqref{Tn} directly, we have
\[ \P(T = n) =  p(n)\, x^n \prod_{i=1}^n (1-x^i).   \]
We would like to maximize this probability, and since this formula holds for all $0<x<1$, we can find an expression for $x$ for which this expression is at its maximum.  It was shown in \cite{Fristedt} (see also \cite{Temperley}) that the choice $x = e^{-c/\sqrt{n}}$, $c = \pi/\sqrt{6}$, is particularly optimal, and produces
\[ \P(T=n) \sim \frac{1}{\sqrt[4]{96}\, n^{3/4}}. \]
Thus, the \WTGL\ algorithm has cost
\[ \WTGLTIME(\X') = O\left(\TIME(\X) \, n^{3/4} \right). \]
It was shown in \cite{PDC} that for any selected index $i\in\{1,\ldots,n\}$, the \PDCTSH\ algorithm obtains a speedup to \WTGL\ of size
\[ \speedup  = \Omega\left(\max_{j\geq 0} \left(x^{i j} (1-x^i) \right)\right)^{-1}  = \Omega \left(1-x^i\right)^{-1} = \Omega\left( \frac{\sqrt{n}}{i\, c} \right). \]
 Thus, the optimal choice is $I = \{1\}$, and the total time for the \PDCTSH\ algorithm is
 \[ \TIMEDDSH(\X') = O\left( \TIME(\X)\, n^{1/4}\right). \]
     
\begin{remark}\label{self-similar}{\rm
It was shown in \cite{PDC} that by using $I = \{1, 3, 5, \ldots\}$, i.e., all odd parts, one can apply a self--similar recursive PDC algorithm for which the total time of the algorithm is $O(\TIME(\X))$; this method is PDC but not \PDCTSH.  The algorithm relies on a bijection and the fact that 
the rejection function can be computed to arbitrary accuracy efficiently using the analysis in~\cite{HR, johansson2012efficient, Lehmer, Rademacher}, and also the log-concavity of the partition function for all $n \geq 26$ due to~\cite{Nicolas}; see also~\cite{logconcave}. 
Thus, for \emph{unrestricted} integer partitions we prefer the self--similar algorithm, but in general such an advantageous structure may not be known. 
}\end{remark}

When partitions are restricted to have distinct parts, then we have  $J = \{1,\ldots, n\}$, $w(i) = i$, $f(J,n) = 1$, $g_i(a_i) = 1(a_i \leq 1)$, $p(n)$ is the number of partitions of $n$ into distinct parts, usually denoted by $q(n)$.  Then we have normalization constants $c_i = (1+x^i)^{-1}$, so
\[\P(Z_i = k) = \frac{x^{i k}}{1+x^{i k}}, \qquad k\in \{0,1\}, \ 0<x<1, \ i=1,\ldots n, \]
i.e., $Z_i$ is Bernoulli with parameter $\frac{x^i}{1+x^i}$.  For any $i$, the speedup is thus
\[ \speedup = \Omega\left(\max_{k\in\{0,1\}} \frac{x^{i\,k}}{1+x^{i\,k}} \right)^{-1}  = \Omega \left( \frac{x}{1+x}\right)^{-1} = \Omega(1). 
\]
Thus, for partitions into distinct parts, \PDCTSH\ offers only a constant factor improvement.

\subsection{Example: Selections}

Integer partitions of size~$n$ into distinct parts is an example of a selection: each element $\{1,2,\ldots,n\}$ is either in the partition or not in the partition.  Selections in general allow $m_i$ different types of a component with weight $i$.  For integer partitions into distinct parts, this would be like assigning $m_i$ colors to integer $i$, and allowing at most one component of size $i$ of each color.  Then we have for all $0<x<1$
\[ \P(Z_i = k) = \binom{m_i}{k} \left(\frac{x^i}{1+x^i}\right)^k\left(\frac{1}{1+x^i}\right)^{m_i-k}, \]
which is binomial.  
The \PDCTSH\ algorithm using $I = \{i\}$ has
\[ \speedup = \Omega\left( \max_i \P\left(Z_i = \frac{m_i x^i}{1+x^i}\right)^{-1}\right) = \Omega\left( \sqrt{\frac{m_i\, x^i}{(1+x^i)^2}} \right)  = \Omega\left( \sqrt{\Var(Z_i)}\right), \] 
where the final two equalities follow since $Z_i$ is a binomial random variable, whence the largest point probability is centered at its expectation, and is approximately $1/\sqrt{2\pi \Var(Z_i)}$.

\subsection{Example: Multisets}

Unrestricted integer partitions of size~$n$ is an example of a multiset: each element $\{1,2,\ldots, n\}$ can appear any number of times in the partition.  Multisets in general allow $m_i$ different types of a component with weight $i$, similar to selections.  We have for all $0<x<1$
\[ \P(Z_i = k) = \binom{m_i+k-1}{k} (1-x^i)^{m_i}x^{i k},\qquad k=0, 1,\ldots, \]
which is negative binomial.  The mode of the negative binomial distribution is given by the mass at $\lfloor (m_i-1) x^i / (1-x^i)\rfloor$.  
Similarly as with selections, we have 
\[ \speedup = \Omega\left( \max_i \P\left(Z_i = \frac{m_i x^i}{1-x^i}\right)^{-1}\right) = \Omega\left( \sqrt{\frac{m_i\, x^i}{(1-x^i)^2}} \right)  = \Omega\left( \sqrt{\Var(Z_i)}\right). \]

\subsection{Assemblies}
\label{sect:assemblies}
Assemblies are described using $Z_i$ as Poisson$(\lambda_i)$, where $\lambda_i = \frac{m_i x^i}{i!}$, $i=1,\ldots,n$, and where $m_i$ denotes the number of components of size $i$, and $x>0$.  
We have
\begin{equation}\label{Poisson prob}
 \P(Z_1 = c_1, \ldots, Z_n = c_n) = \prod_{i=1}^n \frac{ m_i^{c_i} x^{i\, c_i}}{i!^{c_i} c_i!} e^{-\lambda_i} = x^{n} e^{-\sum_{i=1}^n \lambda_i} \prod_{i=1}^n \frac{m_i^{c_i}}{i!^{c_i} c_i!}.
 \end{equation}
 Here again since the random variables are Poisson, the local central limit theorem implies for $\lambda_i$ large that we have 
 \[
  \speedup =  \Omega\left( \sqrt{\Var(Z_i)}\right) = \Omega\left(\sqrt{\lambda_i}\right),
 \]
 hence we should select the index $i$ with the largest variance to obtain the largest speedup.

For set partitions, $C_i$ denotes the number of blocks of size $i$ in a random set partition of size~$n$, $i=1,\ldots,n$.  In this case $m_i = 1$, and we have $\lambda_i = x^i/i!$, for $x>0$.  In this case the probability of the event $\{T=n\}$ is maximized by the value of $x$ such that $x\, e^x = n$, for which $x = \log(n)$ serves as a reasonable approximation for large $n$.\footnote{In \cite{deBruijn}, it is shown that 
\[
x \sim \log(n) - \log(\log(n)) + O\left( \frac{\log \log n}{\log n}\right) .
\]}
It was shown in  \cite{PDC} that $I = [\log(n)]$ is a particularly good choice, since
\[\lambda_I \sim \frac{\log(n)^{\log(n)}}{(\log(n))!} \sim \frac{e^{\log(n)}}{\sqrt{2\pi\log(n)}} \sim \frac{n}{\sqrt{2\pi\log(n)}}, \]
hence,
\[ \speedup = \Omega\left( \sqrt{2\pi\lambda_I}\right) = \Omega\left(\sqrt{n} / \sqrt[4]{\log(n)}\right).  \]
It fact, it was shown in \cite{PittelSetPartitions} that 
\[ \P(T=n) = O\left(\sqrt{n\, \log(n)}\right), \]
whence the total time for \PDCTSH\ is $O(\log^{5/4}(n))$.  
One can even improve upon this rejection rate using the recursive method, see~\cite{DeSalvoImprovements}. 

\subsection{Plane partitions} \label{sect:pp}
We now recall the third example of \cite[Section 3.3.1]{PDC}, which improves upon an algorithm in \cite{BodiniPlane} for random sampling of plane partitions. 
The first step of the algorithm is to sample from a rectangular grid of random variables $\{Z_{i,j}(x)\}_{1\leq i,j\leq n}$, where $Z_{i,j}$ is geometrically distributed with parameter $x^{i+j+1}$, with $x \sim 1 - (2 \xi(3)/n)^{1/3}$ chosen to maximize the probability of the event $E = \left\{\sum_{i,j \geq 1} (i+j+1) Z_{i,j} = n\right\}$. 
\begin{lemma}{\cite[Lemma 7]{BodiniPlane}}
Let $\X = \{Z_{i,j}(x)\}_{1\leq i,j\leq n}$, where $Z_{i,j}$ is geometrically distributed with parameter $x^{i+j+1}$, with $x \sim 1 - (2 \xi(3)/n)^{1/3}$; let $E = \left\{\sum_{i,j \geq 1} (i+j+1) Z_{i,j} = n\right\}$.
Random sampling of $\L(\X)$ can be performed in $O(n^{2/3})$ operations.  
Furthermore, the \WTGL\ algorithm, i.e., the Boltzmann sampler with a hard rejection step, has probability $O(n^{2/3})$ of hitting the target. 
Therefore, sampling of $\L(\X')$ has total $O(n^{4/3})$ operations.
\end{lemma}

To obtain a plane partition, they apply a bijection in~\cite{PakBijection}, which only requires $O(n \log^3(n))$ operations, which makes the entire algorithm on the same order as the sampling of $\L(\X')$, namely, $O(n^{4/3})$. 
We surmise a self-similar PDC sampling algorithm for $\L(\X')$ exists which is $O(n^{2/3}),$ however, for this particular application, this is not necessary. 
Consider the PDC with partition $A = \{Z_{i,j}\}_{i,j\neq 1}$, and $B = \{Z_{1,1}\}$. 
By Equation~\eqref{discrete speedup}, the speedup over \WTGL\ is given by
\[ \speedup = \Omega \left(\max_k \P(Z_{1,1} = k)\right)^{-1} = \Omega\left(\P(Z_{1,1} = 0)\right)^{-1} = \Omega\left( \frac{1}{1-x} \right) = \Omega \left(\frac{n^{1/3}}{2 \xi(3)} \right). \]

Thus, using \PDCTSH, sampling from $\L(\X')$ can be performed in $O(n)$ total operations, which makes the entire algorithm $O(n \log^3(n))$. 

\subsection{Further restrictions on components}
In addition to the condition that $\{\sum_i w_i Z_i = n\}$, one can further condition, e.g.,  on the event $\{\sum_i Z_i = k\}$, which demands that the total number of components in the random structure with weight $n$ is $k$; see, e.g.,~\cite[Section 8]{IPARCS}.  For integer partitions, this is equivalent to sampling from partitions of size~$n$ with exactly $k$ parts; or, set partitions of size~$n$ with exactly $k$ blocks.  
  In general, this condition takes the form $\{\sum_i u_i Z_i = k\}$ for some set of nonnegative coefficients $u = (u_i)_{i \geq 1}$.  Then, for any $\theta>0$, we let $Z_i$ have distribution given by
\[\P_\theta(Z_i = c_i) =  \frac{\theta^{u_i c_i}}{\e \theta^{u_i Z_i}}\P(Z_i = c_i), \qquad i=1,\ldots, n,
\]
where we assume $\e \theta^{u_i Z_i} < \infty$.  Then, by choosing a value of $\theta$, we can effectively tilt the distribution.

For example, take $u_i = 1$ for all $i$ and  $Z_i$ to be Poisson with parameter $\lambda_i = \frac{\theta x^i}{i}$, $i=1,\ldots,n$, and some $\theta>0$.  Then 
\begin{align}\label{Poisson Ewens}
 \P(Z_1 = c_1, \ldots, Z_n = c_n) & = \frac{x^{n} e^{-\theta\sum_{i=1}^n \lambda_i}}{n!}\frac{n!}{\theta(\theta+1)(\theta+2)\ldots(\theta+n-1)} \prod_{i=1}^n \frac{\theta^{c_i}}{i^{c_i} c_i!} \\
\nonumber  & =  \frac{\theta^{k}}{\theta(\theta+1)(\theta+2)\ldots(\theta+n-1)} \frac{x^{n} e^{-\theta\sum_{i=1}^n \lambda_i}}{n!} n! \prod_{i=1}^n \frac{1}{i^{c_i} c_i!}.
 \end{align}
The joint distribution~$\L(Z_1, \ldots, Z_n)$ is the independent process approximation to the Ewens sampling formula, which is given similarly by  $\L((Z_1, \ldots, Z_n) | E),$ where \\ ${E = \left\{ \sum_{i=1}^n i Z_i = n,\ \  \sum_{i=1}^n Z_i = k\right\},}$ see~\cite{Ewens}; see also~\cite[Section 8]{IPARCS}. 
  Letting $t_A := \sum_{i=3}^n i\, Z_i$ and $s_A := \sum_{i=3}^n Z_i$, there is a \PDCTSH\ algorithm for this family; namely, let $I=\{1,2\}$, then we have
\[ E = \left\{ \sum_{i=1}^n i Z_i = n,\ \  \sum_{i=1}^n Z_i = k\right\},\]
\[ \EI = \left\{(x_3,\ldots,x_n) : \sum_{i=3}^n i\, x_i \leq n,\ \ \sum_{i=3}^n x_i \leq k \right\}, \]
\[E_I = E_{I|\XI} = \left\{(y,z) : y+2z+\sum_{i=3}^n i\, Z_i = n,\ \ y+z+\sum_{i=3}^n Z_i = k \right\}, \]
\[ \L\left(\XI\, \big|\, E\right) = \L\left( (Z_3, Z_4, \ldots, Z_n)\, \bigg|\, \sum_{i=1}^n i\, Z_i = n,\ \sum_{i=1}^n Z_i = m\right),\]
\[ \L\left(T_I\, \big|\, \XI=\xI, E\right) = \L\left( (Z_1, Z_2)\, \big|\, Z_1+2Z_2 = n - t_A,\ Z_1 + Z_2 = m - s_A\right).\]
Thus, conditional on accepting a set of values for $\XI = (Z_3, \ldots, Z_n)$, the values of $Z_1$ and $Z_2$, say $y_I$ and $z_I$, are uniquely determined, hence deterministic.  
The rejection condition is then given by
\[ \left\{\text{$\XI \in \EI$ and $U < \frac{\P(Z_1 = y_I)\P(Z_2 = z_I)}{\max_{j_1, j_2} \P(Z_1 = j_1) \P(Z_2 = j_2)}$}\right\},\]
and the speedup is
\[ \speedup = \Omega\left( \left(\max_{j_1, j_2} \P(Z_1 = j_1) \P(Z_2 = j_2)\right)^{-1} \right).  
\]
Again, we strongly note that no local central limit theorems are required, nor any qualitative or quantitative information pertaining to the Ewen's sampling formula or its approximation; the only requirement is the ability to efficiently compute the largest point probability of a given distribution. 

There is also a two--parameter family which generalizes the Ewen's sampling formula given in \cite{Pitman2}.  In general, the same kind of \PDCTSH\ algorithm can be applied to any set of restrictions, as long as they satisfy the regularity conditions, and the choice of $I$ uniquely determines a completion in the second stage of PDC.

\subsection{Logarithmic Combinatorial Structures}\label{sect:lcs}
While \PDCTSH\ is useful in full generality, and when limited information about the sample space is known, Remark~\ref{self-similar} demonstrates that it is possible to substantially improve on the PDC deterministic second half approach for the sampling of $\L(\X')$ in certain cases. 
One particular class of examples is known as \emph{logarithmic combinatorial structures}, see for example~\cite{Logarithmic}, which is the case 
\[ i\, \e X_i \to \kappa, \qquad i\, \P(X_i = 1) \to \kappa. \]
For assemblies, this is equivalent to $m_i \sim \kappa\, y^i (i-1)!$ as $i\to\infty$, for some $y>0$ and $\kappa>0$; for selections and multisets, this is equivalent to $m_i \sim \kappa\, y^i / i$ as $i\to\infty$, for some $y>1$ and $\kappa>0$.

One particular example, which demonstrates the utility of a more detailed understanding of the underlying structure, is for random sampling of permutations according to cycle structure. 
Let $C_i(n)$ denote the number of cycles of length $i$ in a random permutation of $n$, and let $X_i$ be Poisson with parameter $1/i$, $i=1,2,\dots$. 
A classical result, see \cite{Goncharov, Kolchin}, is that as $n$ tends to infinity, we have
\[ (C_1(n), C_2(n), \ldots ) \stackrel{D}{\to} (Z_1, Z_2, \ldots). \]
The Feller coupling is a way to sample from $(C_1(n), C_2(n), \ldots)$ directly; 
see~\cite{FellerCoupling, Renyi}; see also \cite{PPEwens, Logarithmic}.  
Start with 1 by itself in a cycle.  With probability $1/n$, close off the cycle, and start a new cycle with 2 by itself; otherwise, add a uniform number $\{2,\ldots,n\}$ to the cycle with $1.$  Continue this process, closing off the cycle with probability $1/(n-1)$; otherwise, adding a number from the elements remaining uniformly at random to the currently open cycle.
We have
\[ C_i(n) = \#\{\text{exactly $i-1$ times not closing a cycle followed by closing the cycle}\}. \]
According to our costing scheme, which ignores the sizes of the values, this algorithm is $O(n)$, which is asymptotically best possible\footnote{Note that while the entropy lower bound is $\log n! = \Omega(n \log(n))$, this assumes that the complexity of arithmetic manipulations of a number $n$ grows like $O(\log(n))$, whereas we assume this cost is fixed at $O(1)$ for all $n$.}.

This example serves to remind the reader that while \PDCTSH\ offers an improvement to \WTGL, and is often just as practical given little knowledge of the underlying sample space, an improved algorithm can often be fashioned using more sophisticated arguments.

\section{Acknowledgements}
We gratefully acknowledge Richard Arratia, Tom Liggett, and Georg Menz for helpful suggestions, and Igor Pak for introducing us to the area of combinatorial polytopes and help with the literature. 
We also thank an anonymous referee for helpful comments which led to the improvement of Section~\ref{combinatorial classes}.

\bibliographystyle{plain}                                              
\bibliography{../../../master_bib}

\end{document}